\documentclass[english,12pt]{article}

\pdfoutput=1

\usepackage[T1]{fontenc}
\usepackage[latin9]{inputenc}
\usepackage{amssymb}
\usepackage{babel}
\usepackage{amsthm}
\usepackage{verbatim}
\usepackage{amsrefs}
\usepackage{amsmath}
\usepackage{listings}
\usepackage{fullpage}
\usepackage{graphicx}
\usepackage{qtree}
\usepackage[usenames,dvipsnames]{color}
\usepackage{array}
\usepackage{enumerate}
\usepackage{polynom}
\usepackage{tikz}
\usepackage{fancyhdr}
\usepackage{etoolbox}
\usepackage{algorithm}
\usepackage{algorithmic}

\newtheorem{prop}{Proposition}
\newtheorem{theorem}{Theorem}
\newtheorem{lemma}{Lemma}
\newtheorem{cor}{Corollary}

\newcommand{\p}[1]{\left(#1\right)}
\newcommand{\st}[1]{\left\{#1\right\}}

\newcommand{\fl}[1]{\left\lfloor#1\right\rfloor}

\newcommand{\quot}[1]{``#1''}

\newcommand{\limf}[3]{\lim_{#1\rightarrow#2}{#3}}

\newcommand{\limi}[2]{\limf{#1}{\infty}{#2}}

\newcommand{\lb}{\hspace*{\fill}}

\DeclareMathOperator{\modd}{mod}

\newcommand{\fwd}{\p{\Longrightarrow}}
\newcommand{\bwd}{\p{\Longleftarrow}}

\newcommand{\iffpf}[2]{
\begin{description}
\item[$\fwd$] #1
\item[$\bwd$] #2
\end{description}
}

\newcommand{\namehead}[3]{
\lstset{breaklines=true, morecomment=[l]{//}, frame=single, showstringspaces=false, numbers=left}
\begin{flushright}
Nathan Fox\\
#2\\
#3\\
\end{flushright}
\ifstrequal{#1}{.}{}{
\begin{center}
{\Large Homework #1}
\end{center}}
}

\renewcommand{\p}[1]{(#1)}
\newcommand{\pb}[1]{\left(#1\right)}

\newcommand{\seq}{\pb}

\newcommand{\Rh}{B}

\begin{document}
%
%
\title{A Slow Relative of Hofstadter's $Q$-Sequence}
\author{Nathan Fox\footnote{Department of Mathematics, Rutgers University, Piscataway, New Jersey,
\texttt{fox@math.rutgers.edu}
}}
\date{}

\maketitle

\begin{abstract}
Hofstadter's $Q$-sequence remains an enigma fifty years after its introduction.  Initially, the terms of the sequence increase monotonically by $0$ or $1$ at a time.  But, $Q\p{12}=8$ while $Q\p{11}=6$, and monotonicity fails shortly thereafter.  In this paper, we add a third term to Hofstadter's recurrence, giving the recurrence $\Rh\p{n}=\Rh\p{n-\Rh\p{n-1}}+\Rh\p{n-\Rh\p{n-2}}+\Rh\p{n-\Rh\p{n-3}}$.  We show that this recurrence, along with a suitable initial condition that naturally generalizes Hofstadter's initial condition, generates a sequence whose terms all increase monotonically by $0$ or $1$ at a time.  Furthermore, we give a complete description of the resulting frequency sequence, which allows the $n^{th}$ term of our sequence to be efficiently computed.  
We conclude by showing that our sequence cannot be easily generalized.
\end{abstract}

\section{Introduction}

The Hofstadter $Q$-sequence~\cite{geb} is defined by the nested recurrence $Q\p{n}=Q\p{n-Q\p{n-1}}+Q\p{n-Q\p{n-2}}$ with the initial conditions $Q\p{1}=Q\p{2}=1$.  The first $11$ terms of this sequence~\cite[A005185]{oeis} are 
\[
1, 1, 2, 3, 3, 4, 5, 5, 6, 6, 6,\ldots
\]
These terms increase monotonically with successive differences either $0$ or $1$.  But, $Q\p{12}=8$, ending the successive difference property.  Not long thereafter, $Q\p{15}=10$ and $Q\p{16}=9$, ending the monotonicity.  Calculating more terms leads one to the resignation that the Hofstadter $Q$-sequence is anything but well-behaved.  While there appear to be some patterns in the sequence, all such observation are as-yet purely emprical.  Essentially nothing has been rigorously proven about this sequence.  Most critically, nobody has been able to prove that $Q\p{n}$ even exists for all $n$.  If $Q\p{n-1}\geq n$ for some $n$, then evaluating $Q\p{n}$ would require knowing $Q\p{k}$ for some $k\leq0$.  Since $Q$ is only defined for positive indices, $Q\p{n}$ (and all subsequent terms) would fail to exist in this case.  If a sequence is finite because of behavior like this, we say that the sequence \emph{dies}.

Hofstadter and Huber~\cite{hofsem,hofv} investigated the following familiy of recurrences, which generalize the Hofstadter $Q$-recurrence.  For integers $0<r<s$, define
\[
Q_{r,s}\p{n}=Q_{r,s}\p{n-Q_{r,s}\p{n-r}}+Q_{r,s}\p{n-Q_{r,s}\p{n-s}}.
\]
They explored these recurrences experimentally for various initial conditions.  This work led them to conjecture that the sequences reulting from an all-ones initial condition always die, except for $\p{r,s}\in\st{\p{1,1},\p{1,4},\p{2,4}}$.  The case $\p{1,1}$ is the $Q$-sequence, and the case $\p{2,4}$, often called the $W$-sequence, displays even wilder behavior than the $Q$-sequence~\cite[A087777]{oeis}.  The sequence resulting from $\p{r,s}=\p{1,4}$, on the other hand, behaves much more regularly.  This sequence, known as the $V$-sequence, was proven to be monotone increasing by $0$ or $1$ at a time~\cite{hofv}.  This growth property is known in the literature as \emph{slow}.

There has been substantial research concerning slow Hofstadter-like sequences. The most famous example is perhaps the Hofstadter-Conway \$10000 Sequence~\cite[A004001]{oeis}, given by $A\p{n}=A\p{A\p{n-1}}+A\p{n-A\p{n-1}}$ with $A\p{1}=A\p{2}=1$.  Conway notably offered a \$10000 prize for an analysis of the behavior of this sequence.  Colin Mallows solved this problem a few years later~\cite{mallows}.  Another prototypical example is Conolly's~\cite{con} recurrence $C\p{n}=C\p{n-C\p{n-1}}+C\p{n-1-C\p{n-2}}$ with $C\p{1}=C\p{2}=1$ as the initial condition~\cite[A046699]{oeis}.  There are many examples of slow sequences that generalize Conolly's recurrence~\cite{erickson, isgur2}, some of which have combinatorial interpretations involving counting leaves in tree structures~\cite{isgur2}.  In addition, given a slow Hofstadter-like sequence, it is possible to generate an infinite family of slow sequences with similar recurrences~\cite{isgur1}.

Most of the known examples of slow sequences have at least one of the following properties:
\begin{itemize}
\item An inner recursive call with a positive coefficient (like the first $A\p{n-1}$ in the Hofstadter-Conway recurrence).
\item A \quot{shift} in at least one of the recurrence terms (like the $-1$ in the second term in Conolly's recurrence).
\end{itemize}
In fact, the only ones that have neither property are the $V$-sequence and sequences constructed from it~\cite{isgur1}.  
We decided to search for additional slow, Hofstadter-like sequences without these properties.  The investigation of Hofstadter and Huber empirically rules out two-term recurrences, so we began our search by considering the generic $3$-term recurrence
\[
Q_{r,s,t}\p{n}=Q_{r,s,t}\p{n-Q_{r,s,t}\p{n-r}}+Q_{r,s,t}\p{n-Q_{r,s,t}\p{n-s}}+Q_{r,s,t}\p{n-Q_{r,s,t}\p{n-t}}
\]
with integers $0<r<s<t$.  The all-ones initial condition proved fruitless in our investigation.  However, the initial conditions $V\p{1}=1$, $V\p{2}=2$, $V\p{3}=3$, $V\p{4}=4$ generate the $V$-sequence as well (offset by $3$ terms)~\cite{hofv}.  Thus, we focused our search on slow sequences with initial conditions of the form $Q_{r,s,t}\p{i}=i$ for $i\leq t$.  This allowed us to find the sequence with $\p{r,s,t}=\p{1,2,3}$.  In this paper, we prove that this sequence is slow.  In fact, we completely characterize the terms of this sequence and exhibit an efficient algorithm for computing the $n^{th}$ term.  In particular, each term of this sequence appears at most twice, in contrast to the $V$-sequence, whose terms appear at most three times~\cite{hofv}.  
In Section~\ref{sec:seq}, we examine this sequence and prove our results about it.  Then, in Section~\ref{sec:fut}, we discuss some future directions, and we show that one potential generalization of our sequence fails to yield other slow sequences.

\section{Our Sequence}\label{sec:seq}

We will consider the sequence defined by the recurrence
\[
\Rh\p{n}=\Rh\p{n-\Rh\p{n-1}}+\Rh\p{n-\Rh\p{n-2}}+\Rh\p{n-\Rh\p{n-3}}
\]
and the initial conditions $\Rh\p{1}=1$, $\Rh\p{2}=2$, $\Rh\p{3}=3$, $\Rh\p{4}=4$, $\Rh\p{5}=5$.  The first few terms of this sequence (A278055 in OEIS) are
\[
1, 2, 3, 4, 5, 6, 6, 7, 8, 9, 9, 10, 11, 12, 12, 13, 14, 15, 15, 16, 17, 17, 18, 18, 19, 20, 21, 21,\ldots
\]
The main thing we wish to prove is the following:
\begin{theorem}\label{thm:slow}
For all $n$, $\Rh\p{n}-\Rh\p{n-1}\in\st{0,1}$.  In other words, the sequence $\seq{\Rh\p{n}}_{n\geq1}$ is slow.
\end{theorem}
We will actually prove considerably more than just Theorem~\ref{thm:slow}.  We will completely determine the structure of this sequence.  In the terms listed above, each positive integer appears no more than twice (and at least once).  We will show that this is the case for all numbers, and we will completely characterize which numbers repeat.

We will make use of the following auxiliary sequence $\seq{a_i}_{i\geq1}$.  Let $a_1=3$, and for $i\geq1$, let $a_i=3a_{i-1}-1$.  (This is sequence A057198 in OEIS.)  This sequence has the closed form $a_i=\frac{5}{2}3^{i-1}+\frac{1}{2}$.  We have the following theorem.
\begin{theorem}\label{thm:struct}
Let $m$ be a positive integer.  If there exists some integer $k\geq1$ such that $m=k\cdot3^i+a_i$ for some $i\geq1$, then $m$ appears in the $\Rh$-sequence twice.  Otherwise, $m$ appears once.  Furthermore, the $\Rh$-sequence is monotone increasing.
\end{theorem}
Theorem~\ref{thm:struct} implies Theorem~\ref{thm:slow}, since Theorem~\ref{thm:struct} asserts both that the sequence is monotone and that each positive integer appears in the sequence.  Throughout the rest of this section, we will end up proving Theorem~\ref{thm:struct}, and consequently Theorem~\ref{thm:slow}, by induction.  In doing so, we will frequently assume that Theorem~\ref{thm:struct} holds up to some point.  To make this clear, we will define the following indexed families of propositions (where $m$ and $n$ are positive integers):
\begin{itemize}
\item Let $P_m$ denote the proposition \quot{For all integers $1\leq m'\leq m$, if there exists some integer $k\geq1$ such that $m'=k\cdot3^i+a_i$ for some $i\geq1$, then $m'$ appears in the $\Rh$-sequence twice.  Otherwise, $m'$ appears once.  Furthermore, the $
\Rh$-sequence is monotone increasing as long as its terms are at most $m$.}  In this way, $P_m$ is essentially the statement \quot{Theorem~\ref{thm:struct} holds through \emph{value} $m$.}
\item Let $T_n$ denote the proposition \quot{The first $n$ terms of the $\Rh$-sequence are monotone increasing.  Furthermore, for all $m$ appearing as one of these first $n$ terms, if there exists some integer $k\geq1$ such that $m=k\cdot3^i+a_i$ for some $i\geq1$, then $m$ appears in these first $n$ terms twice (unless this second occurrence would be in position $n+1$).  Otherwise, $m$ appears once.}  In this way, $T_n$ is essentially the statement \quot{Theorem~\ref{thm:struct} holds through \emph{index} $n$.}
\end{itemize}
It should be clear from these definitions that the following are equivalent:
\begin{itemize}
\item Theorem~\ref{thm:struct} is true.
\item $P_m$ holds for all $m\geq1$.
\item $T_n$ holds for all $n\geq1$.
\end{itemize}

We will call a pair of positive integers $\p{k,i}$ such that $m=k\cdot3^i+a_i$ a \emph{witness pair} for $m$, and we will call $i$ a \emph{witness} for $m$.  (Theorem~\ref{thm:struct} says that a value $m$ is repeated if and only if it has a witness.)  We will now show that every $m$ has at most one witness.
\begin{lemma}\label{lem:onei}
For any positive integer $m$, there is at most one $i\geq1$ such that $m\equiv a_i\pb{\modd 3^i}$.
\end{lemma}
\begin{proof}
Suppose for a contradiction that, for some integers $i,j\geq1$, $k_1\cdot 3^i+a_i=k_2\cdot 3^{i+j}+a_{i+j}$.  Then
\[
a_{i+j}-a_i=k_1\cdot 3^i-k_2\cdot 3^{i+j}=3^i\p{k_1+k_2\cdot 3^j}.
\]
In particular, $a_{i+j}-a_i$ must be divisible by $3^i$.

But, using the closed form,
\[
a_{i+j}-a_i=\pb{\frac{5}{2}\cdot3^{i+j-1}+\frac{1}{2}}-\pb{\frac{5}{2}\cdot3^{i-1}+\frac{1}{2}}=\frac{5}{2}\pb{3^{i+j-1}-3^{i-1}}=\frac{5}{2}\cdot3^{i-1}\pb{3^j-1}.
\]
This is clearly not divisible by $3^i$, a contradiction.  Therefore, no such $i$ and $j$ can exist, so there is at most one $i\geq1$ such that $m\equiv a_i\pb{\modd 3^i}$, as required.
\end{proof}
For a value $m$, we will examine the number of values less than $m$ that are repeated.  
Define $R\p{m,i}=\max\pb{0, \fl{\frac{m-a_i-1}{3^i}}}$.  This floored quantity counts the 
witness pairs $\p{k,i}$ for numbers less than $m$.  
If $P_m$ holds, then this is also the number of repeated values $m'<m$ with witness $i$.
If we now let
\[
R\p{m}=\sum_{i=1}^\infty R\p{m,i},
\]
we have that $R\p{m}$ is the total number of repeated values less than $m$ 
(provided that $P_m$ holds.)
This sum converges because only the logarithmically many terms with $a_i-1\leq m$ are nonzero.

We now have the following lemmas.
\begin{lemma}\label{lem:mrn}
Let $m$ be a positive integer.  
Suppose $P_{m-1}$ holds.  
Then, $\Rh\p{m+R\p{m}-1}=m-1$, and $\Rh\p{m+R\p{m}}\geq m$.  (In other words $m+R\p{m}-1$ is the last index in our sequence with value at most $m-1$.)
\end{lemma}
\begin{proof}
The number of terms before the first occurrence of a term greater than or equal to $m$ will be at least $m-1$, since each number smaller than $m$ must appear at least once.  The first occurrence of such a term will be \quot{delayed} by $1$ index for every smaller value that is repeated.  The number of such repeated values is $R\p{m}$.  So, there are $m-1+R\p{m}$ terms before the first occurrence of a term greater than or equal to $m$.  This means that the last occurrence of $m-1$ is in position $m+R\p{m}-1$, as required.
\end{proof}
An immediate consequence of Lemma~\ref{lem:mrn} is that $\Rh\p{m+R\p{m}}$ in fact \emph{equals} $m$, provided that 
$P_m$ holds.
\begin{lemma}\label{lem:rmi}
Let $m$ be a multiple of $3$.  If $i\geq2$ is a witness for $m-1$, then $R\p{m,i}=R\!\pb{\frac{m}{3},i-1}+1$.  Otherwise, $R\p{m,i}=R\!\pb{\frac{m}{3},i-1}$.
\end{lemma}
\begin{proof}
The lemma is clearly true if $a_i+1\geq m$, so we can assume without loss of generality that $a_i+1<m$ and thereby ignore the $\max$ in the definition of $R\p{m,i}$ when proving this lemma.

We have
\[
R\p{m,i}=\fl{\frac{m-a_i-1}{3^i}}=\fl{\frac{m}{3^i}-\frac{a_i+1}{3^i}}
\]
and
\[
R\!\pb{\frac{m}{3},i-1}=\fl{\frac{\frac{m}{3}-a_{i-1}-1}{3^{i-1}}}=\fl{\frac{m}{3^i}-\frac{a_{i-1}+1}{3^{i-1}}}.
\]
Since $a_i=\frac{5}{2}\cdot3^{i-1}+\frac{1}{2}$,
\[
\frac{a_i+1}{3^i}=\frac{5}{6}+\frac{1}{2\cdot3^i}
\]
and
\[
\frac{a_{i-1}+1}{3^{i-1}}=\frac{5}{6}+\frac{1}{2\cdot3^{i-1}}.
\]
The first of these definitely smaller, so $R\p{m,i}\geq R\!\pb{\frac{m}{3},i-1}$.  Furthermore, the above fractions differ by $\frac{1}{3^i}$, so $R\p{m,i}\leq R\!\pb{\frac{m}{3},i-1}+1$.

The only way they will not be equal is if there is some integer $\ell$ such that
\[
\frac{m}{3^i}-\frac{a_{i-1}+1}{3^{i-1}}<\ell\leq\frac{m}{3^i}-\frac{a_i+1}{3^i}.
\]
Since the bounds differ by $\frac{1}{3^i}$ and they have common denominator $3^i$, this can only happen if $\ell=\frac{m}{3^i}-\frac{a_i+1}{3^i}$.  This gives that $m-a_i+1=\ell\cdot3^i$, or $m-1=\ell\cdot3^i+a_i$ for some integer $\ell$.  Since $a_i+1<m$, we must have $\ell\geq1$.  So, for $R\p{m,i}=R\!\pb{\frac{m}{3},i-1}+1$, we obtain that $i$ must be a witness for $m-1$, as required.
\end{proof}
\begin{lemma}\label{lem:rmrep}
Let $m$ be a multiple of $3$.  Then,
\[
\frac{m}{3}+R\!\pb{\frac{m}{3}}=\begin{cases}
R\p{m}+1 & \text{if }m-1\text{ has a witness}\\
R\p{m}+2 & \text{if }m-1\text{ does not have a witness}.
\end{cases}
\]
\end{lemma}
\begin{proof}
As a consequence of Lemma~\ref{lem:rmi} and Lemma~\ref{lem:onei},
\[
R\p{m}=\begin{cases}
R\p{m,1}+R\!\pb{\frac{m}{3}} & \text{if }m-1\text{ does not have a witness}\\
R\p{m,1}+R\!\pb{\frac{m}{3}}+1 & \text{if }m-1\text{ has a witness}.
\end{cases}
\]
We also have
\[
R\p{m,1}=\fl{\frac{m-a_1-1}{3}}=\fl{\frac{m-4}{3}}=\frac{m}{3}-2.
\]
Substituting this into the above and rearranging terms gives the required form.
\end{proof}
\begin{lemma}\label{lem:mm3w}
Let $m$ be a multiple of $3$.  Then $m-1$ has a witness if and only if $\frac{m}{3}$ has a witness.
\end{lemma}
\begin{proof}\lb
\iffpf
{
Suppose $m-1=k\cdot3^i+a_i$ for some positive integers $k$ and $i$.  Then, $m=k\cdot3^i+a_i+1$.  But, $a_i=3a_{i-1}-1$, so $m=k\cdot3^i+3a_{i-1}$.  This means that $\frac{m}{3}=k\cdot3^{i-1}+a_{i-1}$, so $i-1$ is a witness for $\frac{m}{3}$.
}
{
Suppose $\frac{m}{3}=k\cdot3^i+a_i$ for some positive integers $k$ and $i$.  Then, $m=3k\cdot3^i+3a_i$.  But, $a_{i+1}=3a_i-1$, so $m=k\cdot3^{i+1}+a_{i+1}+1$.  This means that $m-1=k\cdot3^{i+1}+a_{i+1}$, so $i+1$ is a witness for $m-1$.
}
\end{proof}
\begin{lemma}\label{lem:brs}
Let $m\geq6$ be a multiple of $3$.  Suppose 
$P_{m-1}$ holds.  
Then, if $m-1$ repeats we have $\Rh\p{R\p{m}+1}=\frac{m}{3}$.  If $m-1$ does not repeat we have $\Rh\p{R\p{m}+1}=\frac{m}{3}-1$.  In both cases we have
\[
\begin{cases}
\Rh\p{R\p{m}+2}=\frac{m}{3}\\
\Rh\p{R\p{m}+3}=\frac{m}{3}+1.
\end{cases}
\]
\end{lemma}
\begin{proof}
We will look at the two cases separately.
\begin{description}
\item[$m-1$ repeats:] Then, $m-1$ has a witness.  So, by Lemma~\ref{lem:rmrep}, $R\p{m}+1=\frac{m}{3}+R\!\pb{\frac{m}{3}}$.  By Lemma~\ref{lem:mrn}, $\Rh\p{R\p{m}+1}=\frac{m}{3}$.  Furthermore, by Lemma~\ref{lem:mm3w}, $\frac{m}{3}$ has a witness (and hence repeats), so $\Rh\p{R\p{m}+2}=\frac{m}{3}$ as well.  Since values appear at most twice, we then have $\Rh\p{R\p{m}+3}=\frac{m}{3}+1$, as required.
\item[$m-1$ does not repeat:] Then, $m-1$ does not have a witness.  So, by Lemma~\ref{lem:rmrep}, $R\p{m}+2=\frac{m}{3}+R\!\pb{\frac{m}{3}}$.  By Lemma~\ref{lem:mrn}, $\Rh\p{R\p{m}+2}=\frac{m}{3}$ and $\Rh\p{R\p{m}+1}=\frac{m}{3}-1$.  Furthermore, by Lemma~\ref{lem:mm3w}, $\frac{m}{3}$ has no witness (and hence does not repeat), so $\Rh\p{R\p{m}+3}=\frac{m}{3}+1$.
\end{description}
\end{proof}

We are now ready to prove Theorem~\ref{thm:struct}.
\begin{proof}
The proof will be by induction on $n$, the index in the sequence.  For the base case, observe that each term in the initial condition appears once, and no such term has a witness.

Now, suppose that
$T_{n-1}$ holds, 
and suppose that we wish to show that $\Rh\p{n}=m$ for some $m\geq6$.  Also, suppose that 
$P_{m-1}$ holds.  
There are seven cases to consider, which cover all possibilities.  (Note that no repeated term is congruent to $1$ mod $3$, since $a_1$ is divisible by $3$ and $a_i\equiv2\pb{\modd 3}$ for all $i\geq2$.)
\begin{description}
\item[$m\equiv0\pb{\modd 3}$, first occurrence, $m-1$ not repeated:] In this case, $m-1$ has no witness and, by Lemma~\ref{lem:mrn}, $n=m+R\p{m}$.  We have (using Lemma~\ref{lem:brs})
\begin{align*}
\Rh\p{n}&=\Rh\p{n-\Rh\p{n-1}}+\Rh\p{n-\Rh\p{n-2}}+\Rh\p{n-\Rh\p{n-3}}\\
&=\Rh\p{m+R\p{m}-\pb{m-1}}+\Rh\p{m+R\p{m}-\pb{m-2}}\\
&\hspace{0.2in}+\Rh\p{m+R\p{m}-\pb{m-3}}\\
&=\Rh\p{R\p{m}+1}+\Rh\p{R\p{m}+2}+\Rh\p{R\p{m}+3}\\
&=\pb{\frac{m}{3}-1}+\frac{m}{3}+\pb{\frac{m}{3}+1}\\
&=m,
\end{align*}
as required.
\item[$m\equiv0\pb{\modd 3}$, first occurrence, $m-1$ repeated:] In this case, $m-1$ has a witness and, by Lemma~\ref{lem:mrn}, $n=m+R\p{m}$.  We have (using Lemma~\ref{lem:brs})
\begin{align*}
\Rh\p{n}&=\Rh\p{n-\Rh\p{n-1}}+\Rh\p{n-\Rh\p{n-2}}+\Rh\p{n-\Rh\p{n-3}}\\
&=\Rh\p{m+R\p{m}-\pb{m-1}}+\Rh\p{m+R\p{m}-\pb{m-1}}\\
&\hspace{0.2in}+\Rh\p{m+R\p{m}-\pb{m-2}}\\
&=\Rh\p{R\p{m}+1}+\Rh\p{R\p{m}+1}+\Rh\p{R\p{m}+2}\\
&=\frac{m}{3}+\frac{m}{3}+\frac{m}{3}\\
&=m,
\end{align*}
as required.
\item[$m\equiv0\pb{\modd 3}$, second occurrence, $m-1$ not repeated:] In this case, $m-1$ has no witness and, by Lemma~\ref{lem:mrn}, $n=m+R\p{m}+1$.  We have (using Lemma~\ref{lem:brs})
\begin{align*}
\Rh\p{n}&=\Rh\p{n-\Rh\p{n-1}}+\Rh\p{n-\Rh\p{n-2}}+\Rh\p{n-\Rh\p{n-3}}\\
&=\Rh\p{m+R\p{m}+1-m}+\Rh\p{m+R\p{m}+1-\pb{m-1}}\\
&\hspace{0.2in}+\Rh\p{m+R\p{m}+1-\pb{m-2}}\\
&=\Rh\p{R\p{m}+1}+\Rh\p{R\p{m}+2}+\Rh\p{R\p{m}+3}\\
&=\pb{\frac{m}{3}-1}+\frac{m}{3}+\pb{\frac{m}{3}+1}\\
&=m,
\end{align*}
as required.
\item[$m\equiv0\pb{\modd 3}$, second occurrence, $m-1$ repeated:] In this case, $m-1$ has a witness and, by Lemma~\ref{lem:mrn}, $n=m+R\p{m}+1$.  We have (using Lemma~\ref{lem:brs})
\begin{align*}
\Rh\p{n}&=\Rh\p{n-\Rh\p{n-1}}+\Rh\p{n-\Rh\p{n-2}}+\Rh\p{n-\Rh\p{n-3}}\\
&=\Rh\p{m+R\p{m}+1-m}+\Rh\p{m+R\p{m}+1-\pb{m-1}}\\
&\hspace{0.2in}+\Rh\p{m+R\p{m}+1-\pb{m-1}}\\
&=\Rh\p{R\p{m}+1}+\Rh\p{R\p{m}+1}+\Rh\p{R\p{m}+2}\\
&=\frac{m}{3}+\frac{m}{3}+\frac{m}{3}\\
&=m,
\end{align*}
as required.
\item[$m\equiv1\pb{\modd 3}$:] In this case, $m-1$ is divisible by $3$ and therefore definitely repeats (since $a_1=3$).  This also means that $R\p{m-1}=R\p{m}-1$.  By Lemma~\ref{lem:mrn}, $n=m+R\p{m}$.  We have (using Lemma~\ref{lem:brs})
\begin{align*}
\Rh\p{n}&=\Rh\p{n-\Rh\p{n-1}}+\Rh\p{n-\Rh\p{n-2}}+\Rh\p{n-\Rh\p{n-3}}\\
&=\Rh\p{m+R\p{m}-\p{m-1}}+\Rh\p{m+R\p{m}-\pb{m-1}}\\
&\hspace{0.2in}+\Rh\p{m+R\p{m}-\pb{m-2}}\\
&=\Rh\p{R\p{m}+1}+\Rh\p{R\p{m}+1}+\Rh\p{R\p{m}+2}\\
&=\Rh\p{R\p{m-1}+2}+\Rh\p{R\p{m-1}+2}+\Rh\p{R\p{m-1}+3}\\
&=\frac{m-1}{3}+\frac{m-1}{3}+\pb{\frac{m-1}{3}+1}\\
&=m,
\end{align*}
as required.
\item[$m\equiv2\pb{\modd 3}$, first occurrence:] In this case, $m-2$ is divisible by $3$ and therefore definitely repeats.  This also means that $R\p{m-2}=R\p{m}-1$.  By Lemma~\ref{lem:mrn}, $n=m+R\p{m}$.  We have (using Lemma~\ref{lem:brs})
\begin{align*}
\Rh\p{n}&=\Rh\p{n-\Rh\p{n-1}}+\Rh\p{n-\Rh\p{n-2}}+\Rh\p{n-\Rh\p{n-3}}\\
&=\Rh\p{m+R\p{m}-\p{m-1}}+\Rh\p{m+R\p{m}-\pb{m-2}}\\
&\hspace{0.2in}+\Rh\p{m+R\p{m}-\pb{m-2}}\\
&=\Rh\p{R\p{m}+1}+\Rh\p{R\p{m}+2}+\Rh\p{R\p{m}+2}\\
&=\Rh\p{R\p{m-2}+2}+\Rh\p{R\p{m-2}+3}+\Rh\p{R\p{m-2}+3}\\
&=\frac{m-2}{3}+\pb{\frac{m-2}{3}+1}+\pb{\frac{m-2}{3}+1}\\
&=m,
\end{align*}
as required.
\item[$m\equiv2\pb{\modd 3}$, second occurrence:] In this case, $m$ has a witness, so $R\p{m+1}=R\p{m}+1$.  Also, $R\p{m-2}=R\p{m}-1$.  By Lemma~\ref{lem:mrn}, $n=m+R\p{m}+1$.  We have (using Lemma~\ref{lem:brs})
\begin{align*}
\Rh\p{n}&=\Rh\p{n-\Rh\p{n-1}}+\Rh\p{n-\Rh\p{n-2}}+\Rh\p{n-\Rh\p{n-3}}\\
&=\Rh\p{m+R\p{m}+1-m}+\Rh\p{m+R\p{m}+1-\pb{m-1}}\\
&\hspace{0.2in}+\Rh\p{m+R\p{m}+1-\pb{m-2}}\\
&=\Rh\p{R\p{m}+1}+\Rh\p{R\p{m}+2}+\Rh\p{R\p{m}+3}\\
&=\Rh\p{R\p{m-2}+2}+\Rh\p{R\p{m-2}+3}+\Rh\p{R\p{m+1}+2}\\
&=\frac{m-2}{3}+\pb{\frac{m-2}{3}+1}+\pb{\frac{m+1}{3}}\\
&=m,
\end{align*}
as required.
\end{description}
\end{proof}

We have the following corollary.
\begin{cor}
We have
\[
\limi{n}{\frac{\Rh\p{n}}{n}}=\frac{2}{3}.
\]
\end{cor}
\begin{proof}
If $\Rh\p{n}=m$, then $n=m+R\p{m}$ or $n=m+R\p{m}+1$.  So, it will suffice to show that
\[
\limi{m}{\frac{m}{m+R\p{m}}}=\frac{2}{3},
\]
for which it is sufficient to show that
\[
\limi{m}{\frac{R\p{m}}{m}}=\frac{1}{2}.
\]
For each $i\geq1$, we have
\[
\limi{m}{\frac{R\p{m,i}}{m}}=\frac{1}{3^i}.
\]
So,
\begin{align*}
\limi{m}{\frac{R\p{m}}{m}}&=\limi{m}{\frac{1}{m}\sum_{i=1}^\infty R\p{m,i}}\\
&=\sum_{i=1}^\infty\limi{m}{\frac{R\p{m,i}}{m}}\\
&=\sum_{i=1}^\infty\frac{1}{3^i}\\
&=\frac{1}{2},
\end{align*}
as required.
\end{proof}

\subsection{Algorithm for Computing the Sequence}

Theorem~\ref{thm:struct} leads to an efficient algorithm for calculating $\Rh\p{n}$.  Observe that, for each $m$ and $i$, $R\p{m,i}$ can be computed efficiently.  Since only logarithmically many terms in the sum for $R\p{m}$ are nonzero, this means that $R\p{m}$ can be computed efficiently.

To compute $\Rh\p{n}$, we seek an $m$ such that $n=m+R\p{m}$.  It may be the case that no such $m$ exists, in which case we need to be able to say that no such $m$ exists, and we need to find $m$ such that $n=m+R\p{m}+1$.  This task can be done efficiently using a binary search.  We know that $\Rh\p{n}\leq n$, so for an initial upper bound on $m$ we can use $n$ (and we can use $1$ as a lower bound).  So, in at most $O\p{\log\p{n}}$ steps, we can either find an $m$ so that $n=m+R\p{m}$ or show that none exists.  In the latter case, the final lower bound we find for $m$ will be such that $n=m+R\p{m}+1$.  The total running time of this algorithm is $O\p{\log^2\p{n}}$.

\section{Beyond Our Sequence}\label{sec:fut}
According to the work of Isgur et al.~\cite{isgur1}, our $\Rh$-sequence is the fundamental member of an infinite family of slow sequences with similar recurrences.  (The next one satisfies the recurrence
$
\Rh'\p{n}=\Rh'\p{n-\Rh'\p{n-2}}+\Rh'\p{n-\Rh'\p{n-4}}+\Rh'\p{n-\Rh'\p{n-6}}
$.) 
As mentioned in the introduction, this family and the family resulting from the $V$-sequence comprise the only known examples of slow Hofstadter-like sequences with all recurrence terms of the form $D\p{n-D\p{n-i}}$ for some $i$.  The author has conducted a search for other such sequences without finding another (nontrivial) example.  An obvious idea would be to generalize the $\Rh$-recurrence to the $k$-term recurrence
\[
\Rh_k\p{n}=\sum_{i=1}^k\Rh_k\p{n-\Rh_k\p{n-i}}
\]
(where $\Rh_3$ is the $\Rh$-recurrence and $\Rh_2$ is the $Q$-recurrence).  If $k=1$, the initial condition $\Rh_1\p{1}=1$ generates the all-ones sequence, which, while technically slow, is not particularly interesting.  Unfortunately, we have the following result:
\begin{theorem}\label{thm:nope}
The $\Rh$-sequence is the only nontrivial slow sequence resulting from a recurrence $\Rh_k$ with an initial condition of the form $\Rh_k\p{i}=i$ for all $i\leq N$ for some $N$.
\end{theorem}
The bulk of the Theorem~\ref{thm:nope} follows from the following proposition:
\begin{prop}\label{prop:jump}
Suppose $k\geq4$.  The sequence generated by the recurrence $\Rh_k$ with the initial condition $\Rh_k\p{i}=i$ for all $i<\frac{k^2+k}{2}$ satisfies
\[
\Rh_k\!\pb{\frac{1}{2}k^3+\frac{1}{2}k^2+2k+1}=\Rh_k\!\pb{\frac{1}{2}k^3+\frac{1}{2}k^2+2k}+2.
\]
In particular, the sequence has a jump of difference $2$, so it is not slow.
\end{prop}
\begin{proof}
For simplicity of notation, let $N=\frac{k^2+k}{2}$.  We will now show that, for
$1\leq r\leq k+1$ and $-k\leq qk+r<\p{N-k}\p{k+1}$
\[
\Rh_k\p{N+q\p{k+1}+r}=N+qk+r-1.
\]
We observe that the last $k$ terms of the initial condition correspond to $q=-1$ and $r=1$ through $r=k$.  These all satisfy $\Rh_k\p{N+q\p{k+1}+r}=N+qk+r-1$, as required.  We also have, when $q=-1$ and $r=k+1$,
\[
\Rh_k\p{N}=\sum_{i=1}^k\Rh_k\p{N-\Rh_k\p{N-i}}=\sum_{i=1}^k\Rh_k\p{N-\p{N-i}}=\sum_{i=1}^k\Rh_k\p{i}=\sum_{i=1}^k i=N,
\]
as required.

Now, let $q\p{k+1}+r>0$ and suppose inductively that
\[
\Rh_k\p{N+q'\p{k+1}+r'}=N+q'k+r'-1
\]
for all $-k\leq q'\p{k+1}+r'<q\p{k+1}+r<\p{N-k+1}\p{k+1}$.  We have
\begin{align*}
\Rh_k\p{N+q\p{k+1}+r}&=\sum_{i=1}^k\Rh_k\p{N+q\p{k+1}+r-\Rh_k\p{N+q\p{k+1}+r-i}}\\
&=\sum_{i=1}^{r-1}\Rh_k\p{N+q\p{k+1}+r-\Rh_k\p{N+q\p{k+1}+r-i}}\\&\hspace{0.2in}+\sum_{i=r}^k\Rh_k\p{N+q\p{k+1}+r-\Rh_k\p{N+q\p{k+1}+r-i}}\\
&=\sum_{i=1}^{r-1}\Rh_k\p{N+q\p{k+1}+r-\pb{N+qk+r-i-1}}\\&\hspace{0.2in}+\sum_{i=r}^k\Rh_k\p{N+q\p{k+1}+r-\pb{N+\p{q-1}k+\p{r-i+k+1}-1}}\\
&=\sum_{i=1}^{r-1}\Rh_k\p{q+i+1}+\sum_{i=r}^k\Rh_k\p{q+i}.
\end{align*}
Since $q\leq N-k$ (and, if $q=N-k$, then $r\leq k$), this equals
\begin{align*}
\sum_{i=1}^{r-1}\p{i+q+1}+\sum_{i=r}^k\p{i+q}=r-1+qk+\sum_{i=1}^k i=N+qk+r-1,
\end{align*}
as required.



Now, for simplicity of notation, let $A=\pb{N-k+1}\pb{k+1}$.  We will now show that, for $0\leq r\leq k-2$, $\Rh_k\p{N+A+r}=N+\pb{N-k+1}k+r-1$.  (Note that these values are $1$ less than they would be if the previous pattern continued.)  Inductively, suppose this holds for all $r'<r$.  
We now calculate
\begin{align*}
\Rh_k\p{N+A+r}&=\sum_{i=1}^k\Rh_k\p{N+A+r-\Rh_k\p{N+A+r-i}}\\
&=\sum_{i=1}^r\Rh_k\p{N+A+r-\Rh_k\p{N+A+r-i}}\\
&\hspace{0.2in}+\sum_{i=r+1}^k\Rh_k\p{N+A+r-\Rh_k\p{N+A+r-i}}\\
&=\sum_{i=1}^r\Rh_k\p{N+A+r-\pb{N+\pb{N-k+1}k+r-i-1}}\\
&\hspace{0.2in}+\sum_{i=r+1}^k\Rh_k\p{N+A+r-\pb{N+\pb{N-k}k+\pb{k+1+r-i}-1}}\\
&=\sum_{i=1}^r\Rh_k\p{A+i+1-\pb{N-k+1}k}+\sum_{i=r+1}^k\Rh_k\p{A+i-\pb{N-k+1}k}\\
&=\sum_{i=1}^r\Rh_k\p{N-k+i+2}+\sum_{i=r+1}^k\Rh_k\p{N-k+i+1}\\
&=N+\sum_{i=1}^r\pb{N-k+i+2}+\sum_{i=r+1}^{k-1}\pb{N-k+i+1}\\
&=Nk-\pb{k-1}k+r+\pb{k-1}+\frac{k^2-k}{2}\\
&=\pb{N-k+1}k+r-1+\pb{k+\frac{k^2-k}{2}}\\
&=N+\pb{N-k+1}k+r-1,
\end{align*}
as required.  The above calculation is also valid for $r=k-1$, except that $\Rh_k\p{N-k+i+2}$ would be $\Rh_k\p{N+1}$ when $i=k-1$.  Recall that $\Rh_k\p{N+1}=N$, rather than $N+1$.  So, we obtain $\Rh_k\p{N+A+k-1}=N+\p{N-k+1}k+k-3$.

We now compute
\begin{align*}
\Rh_k\p{N+A+k}&=\sum_{i=1}^k\Rh_k\p{N+A+k-\Rh_k\p{N+A+k-i}}\\
&=\Rh_k\p{N+A+k-\Rh_k\p{N+A+k-1}}\\
&\hspace{0.2in}+\sum_{i=2}^k\Rh_k\p{N+A+k-\Rh_k\p{N+A+k-i}}\\
&=\Rh_k\p{N+A+k-\pb{N+\pb{N-k+1}k+k-3}}\\
&\hspace{0.2in}+\sum_{i=2}^k\Rh_k\p{N+A+k-\pb{N+\pb{N-k+1}k+k-i-1}}\\
&=\Rh_k\p{A-\pb{N-k+1}k+3}+\sum_{i=2}^k\Rh_k\p{A-\pb{N-k+1}k+i+1}\\
&=\Rh_k\p{N-k+4}+\sum_{i=2}^k\Rh_k\p{N-k+i+2}\\
&=\Rh_k\p{N-k+4}+\sum_{i=2}^{k-2}\Rh_k\p{N-k+i+2}+\Rh_k\p{N+1}+\Rh_k\p{N+2}\\
&=N-k+4+\sum_{i=2}^{k-2}\pb{N-k+i+2}+2N+1\\
&=3N-k+5+\pb{k-3}N-k\p{k-3}+2\pb{k-3}+\pb{\frac{\pb{k-2}\pb{k-1}}{2}-1}\\
&=Nk-k+4-k\pb{k-3}+2\pb{k-3}+\frac{\pb{k-2}\pb{k-1}}{2}\\
&=Nk-k+4-k\pb{k-1}+2k+2k-6+\pb{N-k-\pb{k-1}}\\
&=N+\pb{N-k+1}k+k-1.
\end{align*}
(Observe that these calculations are only valid because $k\geq4$, as otherwise $N-k+4$ would be larger than $N$.)  So, we have $\Rh_k\p{N+A+k}=\Rh_k\p{N+A+k-1}+2$.  Recalling the values of $N$ and $A$, we have that $N+A+k=\frac{1}{2}k^3+\frac{1}{2}k^2+2k+1$, as required.
%
\end{proof}
We will now complete the proof of Theorem~\ref{thm:nope}.
\begin{proof}
Fix a positive integer $N$.  Consider the recurrence $\Rh_k$ with the initial condition $\Rh_k\p{i}=i$ for $1\leq i\leq N$.  Suppose that the sequence we obtain is slow.  Clearly, we need $N\geq k$, or else $\Rh_k\p{N+1}$ is undefined.  Supposing that $N\geq k$, we have
\begin{align*}
\Rh_k\p{N+1}&=\sum_{i=1}^k\Rh_k\p{N+1-\Rh_k\p{N+1-i}}\\
&=\sum_{i=1}^k\Rh_k\p{N+1-\pb{N+1-i}}\\
&=\sum_{i=1}^k\Rh_k\p{i}\\
&=\sum_{i=1}^k i\\
&=\frac{k^2+k}{2}.
\end{align*}
So, unless $N\in\st{\frac{k^2+k}{2}-1, \frac{k^2+k}{2}}$, we would not have $\Rh_k\p{N+1}-\Rh_k\p{N}\in\st{0,1}$.  According to Proposition~\ref{prop:jump}, $N=\frac{k^2+k}{2}-1$ does not result in a slow sequence for $k\geq4$.  Similarly, $N=\frac{k^2+k}{2}$ does not result in a slow sequence for $k\geq4$, as this sequence will be identical to the one for $N=\frac{k^2+k}{2}-1$ (since the first $N$ terms are the same).  So, we must have $N\leq 3$.  The case $N=1$ results in a trivial sequence, $N=2$ give the Hofstadter $Q$-sequence (which is not slow), and $N=3$ gives our $\Rh$-sequence.  Therefore, the $\Rh$-sequence is the only nontrivial slow sequence resulting from a recurrence $\Rh_k$ with an initial condition of the form $\Rh_k\p{i}=i$ for all $i\leq N$ for some $N$, as required.
\end{proof}

\section*{Acknowledgements}
I would like to thank Dr. Doron Zeilberger of Rutgers University for introducing me to this area of study and for providing me with useful feedback on this work.  I would also like to thank Richard Voepel of Rutgers University for proofreading a draft of this paper and providing me with useful feedback.

\bibliography{bibliography.bib}
\end{document}